\documentclass[preprint,12pt]{article}

\usepackage{amscd}

\usepackage{mathrsfs}
\usepackage{amsfonts}
\usepackage{amssymb}
\usepackage{pifont}
\usepackage{amsmath}
\usepackage{mathrsfs}
\usepackage{relsize}
\usepackage{amsmath}
\usepackage{mathrsfs}
\usepackage{amsfonts}
\usepackage{amsmath}
\usepackage{amscd}
\usepackage[all]{xy}

\usepackage{amsthm}
\usepackage{amsmath}
\usepackage{mathrsfs}
\usepackage{amsfonts}
\usepackage{mathrsfs}
\usepackage{amsfonts}
\usepackage{mathrsfs}
\usepackage{bbding}
\usepackage{amsfonts}
\usepackage{tikz}
\usetikzlibrary{%
  calc,
  fit,
  shapes,
  backgrounds
}
\usepackage{tikz-cd}
\usetikzlibrary{calc,matrix,patterns,shadings,backgrounds,fit}
\pgfdeclarelayer{myback}
\pgfsetlayers{myback,background,main}
\usepackage{stmaryrd}

\renewcommand\appendix{\par
  \setcounter{section}{0}
  \setcounter{subsection}{0}
  \setcounter{figure}{0}
  \setcounter{table}{0}
  \renewcommand\thesection{Appendix \Alph{section}}
  \renewcommand\thefigure{\Alph{section}\arabic{figure}}
  \renewcommand\thetable{\Alph{section}\arabic{table}}
}

\usepackage{txfonts}
\usepackage{latexsym}
\usepackage{amssymb}
\usepackage{amstext}
\usepackage{fancyhdr}
\usepackage{graphicx,picinpar,epsf}
\usepackage{lastpage}
\usepackage{hyperref}
\usepackage{verbatim}

\makeatletter
\def\author#1{\gdef\autrun{\def\and{\unskip, }#1}\gdef\@author{#1}}
\def\address#1{{\def\and{\\\hspace*{18pt}}\renewcommand{\thefootnote}{}%
\footnote {#1}}%
\markboth{\autrun}{\titrun}}
\makeatother

\def\keywords#1{\par\medskip
\noindent\textbf{Keywords.} #1}

\newcommand{\ch}{\mathrm{ch}}

\newcommand{\pp}{\mathbf {P}^2}

\newcommand{\Hom}{\mathrm{Hom}}

\newcommand{\ext}{\mathrm{ext}}
\newcommand{\Ext}{\mathrm{Ext}}

\newcommand{\Coh}{\mathrm{Coh}}
\newcommand{\Arg}{\mathrm{Arg}\text{ }}

\newcommand{\cccp}{\{1,\frac{v_1}{v_0},\frac{v_2}{v_0}\}\text{-plane}}

\newcommand{\bi}{q>\frac{1}{2}s^2}
\def\Ker{\mathrm{Ker}}

\renewcommand\appendix{\par
  \setcounter{section}{0}
  \setcounter{subsection}{0}
  \setcounter{figure}{0}
  \setcounter{table}{0}
  \renewcommand\thesection{Appendix \Alph{section}}
  \renewcommand\thefigure{\Alph{section}\arabic{figure}}
  \renewcommand\thetable{\Alph{section}\arabic{table}}
}

\def\H{\ensuremath{\mathbb{H}}}

\def\cH{\ensuremath{\mathcal H}}

\newtheorem{theorem}{Theorem}[section]
\newtheorem{defn}[theorem]{Definition}

\newtheorem{lemma}[theorem]{Lemma}
\newtheorem{rem}[theorem]{Remark}

\newtheorem{pd}[theorem]{Proposition and Definition}

\newtheorem*{theorem*}{Theorem}

\usepackage{amssymb}

\title{Smoothness and Poisson structures of Bridgeland moduli spaces on Poisson surfaces}
\author{Chunyi Li and Xiaolei Zhao}
\date{\today}

\begin{document}
\maketitle
\begin{abstract}
Let $X$ be a projective smooth holomorphic Poisson surface, in other words, whose anti-canonical divisor is effective. We show that moduli spaces of certain Bridgeland stable objects on $X$ are smooth. Moreover, we construct Poisson structures on these moduli spaces.
\end{abstract}

\section*{Introduction}
It is proved by Mukai in \cite{Muk} that the moduli space of stable sheaves on an abelian or a projective K3 surface is smooth and has a natural symplectic structure. This construction has been generalized in two directions. On the one hand, the symplectic structure can be generalized to (holomorphic) Poisson structures. In the paper \cite{Tyu}, the author showed that a Poisson structure on the surface will naturally determine an antisymmetric bivector field on the moduli space of stable sheaves. Bottacin \cite{Bo} then proved that such a bivector field satisfies the closure condition and endows the moduli space with a natural Poisson structure.

On the other hand, instead of coherent sheaves, one may consider moduli spaces of objects in D$^b(X)$, the bounded derived category of coherent sheaves on the surface $X$. These moduli spaces attract much attention recent years, mainly based on the development of Bridgeland stability conditions. Among many applications, these moduli spaces provide interesting birational models of moduli spaces of sheaves. Generalizing Mukai's result, Inaba \cite{In} proved that when $X$ is an abelian or a projective K3 surface, the moduli space of objects $E$ in D$^b(X)$ satisfying Ext$^{-1}(E, E) = 0$ and Hom$(E, E) =\mathbb C$ is smooth and can be equipped with a symplectic structure, hence a holomorphic symplectic manifold.

In this paper we provide a unified generalization of these two directions.

\begin{theorem*}[Theorem \ref{thm:smoothness} and \ref{thm:poissonclosure}]
For a smooth projective surface $X$ equipped with a Poisson structure $s \in H^0(X, -K_X)$, the moduli space of stable objects\footnote{Strictly speaking, we only consider the Bridgeland stability condition given by the tilting construction with respect to a polarization $H$ parallel to $K_X$. See Definition \ref{defn:geostabsq} for details.} in D$^b(X)$ is smooth and is endowed with a Poisson structure $\theta_s$ induced by $s$.
\label{maintro}
\end{theorem*}

There are two new features in our theorem. First, Inaba's smoothness result only requires Ext$^{-1}(E, E) = 0$ and Hom$(E, E) =\mathbb C$. However, in our situation there is no natural numerical condition on the objects to guarantee the smoothness of the moduli spaces. Instead, we need to work with Bridgeland stable objects in an essential way. Note that for a stable object $E$, $E\otimes K_X$ may not be stable with respect to the same stability condition. So different from the sheaf case, the smoothness of moduli of stable objects does not directly follow from Serre duality and slope comparison. Our method generalizes our earlier work \cite{LZ1, LZ2} on $\pp$, but avoids the use of full strong exceptional collections, which exist for $\pp$ but not necessarily for general Poisson surfaces. The current method is suggested by Arend Bayer.

Second, in order to check the closure condition of the Poisson structures, \cite{Bo} reduced the question to an open dense subset parametrizing locally free sheaves. For moduli spaces of stable objects, such open sets may not exist. Instead we compute the deformation theory of objects in terms of complexes, and show the closure condition.\\

\textbf{Future work.} In a series of celebrated works \cite{BM,BM2}, the authors prove that the minimal model program of the moduli space of coherent sheaves on projective K3 surfaces can be run on the space of Bridgeland stability conditions via wall-crossing. One of the main technical point in the work is the so called positivity lemma, i.e., to show that each moduli space of Bridgeland stable objects carries a canonically nef divisor. An analogue result is also achieved for Abelian surfaces in \cite{Yos,MYY} by showing the positivity lemma using Fourier-Mukai transforms. 

Generalizing these results to other surfaces becomes difficult. Besides the positivity lemma, it involves at least two extra difficulties. First, it is not clear in general whether the moduli space still behaves nicely after wall-crossing. For example, higher dimensional component may appear after wall-crossing, and this leads to reducible moduli spaces with bad singularities. Secondly, it is not known in general for which Chern classes there exist stable sheaves. This makes it hard to decide when the moduli space is non-empty, and to give a criterion on the actual walls for the moduli spaces.

Based on previous work \cite{ABCH,CHW,DP}, we solve these problems, and generalize the result in \cite{BM,BM2} to the projective plane in \cite{LZ1,LZ2}. The next natural step is to consider Poisson surfaces. In particular, this paper treats with the first difficulty as mentioned above, and it is the starting point of future study on the MMP for moduli spaces of sheaves on a Poisson surface via wall-crossing.

In another direction of a slightly different flavor, \cite{Hi} provides a systematic way to deform the complex structure on a holomorphic Poisson variety. In the case of moduli spaces of sheaves on a Poisson surface $X$, these Poisson deformations produce new varieties that can be realized as moduli spaces of objects on a `non-commutative' surface. In the ideal cases, stability conditions exist for these `non-commutative' surfaces, and one can run MMP for these deformed moduli spaces via wall-crossing. The models appearing in this procedure are expected to correspond to the Poisson deformations of moduli spaces of Bridgeland stable objects on the original surface $X$, with respect to the Poisson structures we construct in this paper. An example of this appeared in \cite{LZ1}, and our result in the current paper can be used to study the general case.\\

\textbf{Acknowledgments.} The authors are greatly indebted to Arend Bayer for his tremendous assistance. In particular, Lemma \ref{ph} and \ref{lemma:extvanishing} are suggested by him, and first appear in his talk in the workshop ``Derived Categories and Moduli Spaces" at University of Stavanger.  We are grateful to  Wanmin Liu and Emanuele Macr\`{\i} for helpful conversations. Chunyi Li is supported by ERC starting grant no. 337039 ``WallXBirGeom''.

 We would like to thank Sergey Mozgovoy for pointing out a mistake in the first and journal version of this paper. In the previous version, our main result was stated for every polarization $H$. However, the equation \eqref{eq2.3}  fails in general if $H$ is not parallel to $K_X$. As a consequence, the proof of Lemma \ref{lemma:extvanishing} is not valid and we do not see an easy way to fix it for all $H$. Our result only holds for $H$ that is numerically parallel to $K_X$.

\section*{Notation}
Throughout the paper we will work over the complex number field $\mathbb C$. All results may hold for algebraically closed field of characteristic $0$. The only necessary change in the argument is in the last section, where analytic neighborhood should be replaced by small {\'e}tale sites. We will leave this for the readers to check.

\section{Stability conditions}
\subsection{Geometric stability conditions}

In this section we introduce Bridgeland stability conditions on surfaces. Let $(X,H)$ be a polarized smooth projective surface, where $H$ is an ample divisor on $X$. Let $D\in $Pic$_{\mathbb R}(X)$ be a divisor with real coefficient satisfying $H.D=0$. Consider the bounded derived category $\mathrm D^b(X)$ of coherent sheaves on $X$. 

Now let $V$ be a three dimensional real vector space. Denote the Grothendieck group of $\mathrm D^b(X)$ by K$(X)$, and the twisted Chern character $\exp(-D).\ch$ by $\ch^D$. For any object $F$ in $\mathrm D^b(X)$, we write
\[\tilde{v}(F)=(v_0(F),v_1(F),v_2(F)):=\big(H^2.\ch_0^D(F),H.\ch_1^D(F),\ch_2^D(F)\big).\]
This defines a map $\tilde{v}: \mathrm{K}(X)\to V$.

Now we consider the real projective space $\mathbf P(V)$ with homogeneous coordinate $[v_0,v_1,v_2]$, we view the locus $v_0=0$ as the line at infinity. The complement forms an affine real plane, which is referred to as the $\cccp$. We call $\mathbf P(V)$ the projective $\cccp$. When $\tilde{v}(F)\neq 0$, we use $v(F)$ to denote the corresponding point in the projective $\cccp$. In particular, when $\ch_0(F)\neq 0$, $v(F)$ is in the (affine) $\cccp$.

\begin{rem}
In this article, in all arguments on the $\cccp$,
we assume the $\frac{v_1}{v_0}$-axis to be horizontal and the
$\frac{v_2}{v_0}$-axis to be vertical. The term `above' means `$\frac{v_2}{v_0}$ coordinate is greater than'. Other terms such as `below', `to the right' and `to the left' are understood in the similar way. In this paper we always use $l$ to denote a line segment and $L$ to denote a line. 

\label{rem:cccp}
\end{rem}

Now, we follow \cite{BM, Bri08} and recall the construction of geometric stability conditions on $X$.




\begin{defn}
A stability condition $\sigma$ on $\mathrm D^b(X)$ is called
\emph{geometric} if it satisfies the support property and all
skyscraper sheaves $k(x)$ are $\sigma$-stable of the same phase. \label{defn:stabgeo}
\end{defn}

For a torsion-free sheaf $F$, recall that the (H-)slope of $F$ is given by $\frac{H.\ch_1^D(F)}{H^2.\ch_0^D(F)}=\frac{H.\ch_1(F)}{H^2.\ch_0(F)}$. Fix a real number $s$, a torsion pair of coherent sheaves on $X$ is given by:
\begin{itemize}
\item[] $\Coh_{\leq s}$: subcategory of $\Coh(X)$ generated by
$H$-semistable sheaves of slope $\leq s$.
\item[] $\Coh_{>
s}$: subcategory of $\Coh(X)$ generated by $H$-semistable
sheaves of slope $> s$ and torsion sheaves.
\end{itemize}
We may define the tilting heart $\Coh_{\# s}$ $:=$ $\langle\Coh_{\leq s}[1]$, $\Coh_{>
s}\rangle$.

\begin{pd}
For two real number $s$ and $q$ such that $\bi$, there exists a geometric stability condition
$\sigma_{s,q}:=(Z_{s,q},\Coh_{\#s})$ on $\mathrm D^b(X)$, where the central charge is given by
\[ Z_{s,q}(E):=(-v_2(E)+q\cdot v_0(E))+ i(v_1(E)-s\cdot v_0(E)).\]
In this case, $\Ker(Z_{s,q})$ consists of the characters corresponding to the point $(1,s,q)$ in the projective $\cccp$. We write $\phi_{s,q}$ for the phase function of $\sigma_{s,q}$.
\label{defn:geostabsq}
\end{pd}

For the proof that $\sigma_{s,q}$ is indeed a geometric stability condition, we refer to \cite{BM} Corollary 4.6 and \cite{Bri08}. Here the phase function $\phi_{s,q}$ can be also defined for objects in
$\Coh_{\#s}$: \[\phi_{s,q}(E) := \frac{1}{\pi} \Arg(Z_{s,q}(E)).\] It is
well-defined in the sense that it coincides with the phase function on
$\sigma_{s,q}$-semistable objects.

\begin{rem}
The definition of  $\sigma_{s,q}$ here, though appears in a different from, is essentially the same as the usual one such as that in \cite{ABCH}. We refer to \cite{LZ2} for a detailed comparison between these two setups.

\label{rem:anglephase}
\end{rem}

\begin{rem}
Given a point $P=(1,s,q)$ with $\bi$, we will also write $\sigma_P$, $\Coh_{P}(X)$ and $Z_P$ for the stability condition $\sigma_{s,q}$, the tilt heart $\Coh_{\#s}(X)$ and the central charge $Z_{s,q}$ respectively. 
\label{rem:sigmap}
\end{rem}

\subsection{Potential walls and phases}\label{sec1.3}
We collect some well-known and useful results about the potential walls in this section. First we have the following description of the potential wall, i.e. the locus of stability conditions for which two given characters are of the same slope.

\begin{lemma}[Potential walls]
 Let $P=(1,s,q)$ be a point with $\bi$; $E$ and $F$ be two objects in $\mathrm{Coh}_{P}(X)$ such that $\tilde{v}(E)$ and $\tilde{v}(F)$ are not zero, then
\begin{center}
complex numbers $Z_P(E)$ and $Z_P(F)$ are on the same ray
\end{center}
if and only if
\begin{center}
$v(E)$, $v(F)$ and $P$ are collinear in the projective $\cccp$.
\end{center}
If $F$ is a $\sigma_P$-stable object, then two potential walls of it do not intersect in the region $\bi$, unless they are identical.
\label{lemma:paraandspanplane}\end{lemma}
\begin{proof}
$Z_P(E)$ and $Z_P(F)$ are on the same ray if and only if $Z_P(a\tilde{v}(E)-b\tilde{v}(F))$  $=$ $0$ for some
$a,b\in\mathbb R_+$. This happens only when $v(E)$, $v(F)$ and $\Ker Z_P$ are collinear in the projective $\cccp$.

For the second statement, note that by the Bogomolov inequality for $\sigma_P$-stable objects, we have
\[
(H^2.\ch_0^D(F))\cdot\ch_2^D(F)\leq (H.\ch_1^D(F))^2.
\]
So $v(F)$ is in the region $q \leq \frac{1}{2}s^2$, and by the first statement, this is the only intersection point of potential walls of $F$.
\end{proof}

Note that this statement holds even when $E$, $F$ are torsion, i.e. $\ch_0=0$. The second statement is first observed by Bertram, and appears in print in \cite{Maci}.

We make some notations for lines and rays on the (projective) $\cccp$. Consider objects $E$ and $F$ such that $\tilde v(E)$ and $\tilde v(F)$ are not zero, and let $\sigma_{s,q}=\sigma_P$ be a geometric stability condition. Let $L_{EF}$ be the straight line on the projective $\cccp$ across $v(E)$    and $v(F)$. $L_{EP}$, also denoted as $L_{E\sigma}$, is the line across $v(E)$ and $P$. We use $l_{EF}$ to denote the line segments on the $\cccp$ when $v(E)$ and $v(F)$ are not at infinity. 

$\mathcal H_P$ is the right half plane with either $\frac{v_1}{v_0}>s$, or $\frac{v_1}{v_0}=s$ and $\frac{v_2}{v_0}>q$. $l_{PE}^+$ is the ray along $L_{PE}$ starting from $P$ and completely contained in $\mathcal H_P$. $L_{E\pm}$ is the vertical wall $L_{E(0,0,1)}$. $l_{E+}$ is the vertical ray along $L_{E(0,0,1)}$ from $E$ going upward. $l_{E-}$ is the vertical ray along $L_{E(0,0,-1)}$ from $E$ going downward.

\begin{lemma}
Let $P = (1,s,q)$ be a point with $\bi$, $E$ and
$F$ be two objects in $\Coh_{\#s}$. The inequality
\[\phi_{s,q}(E)>\phi_{s,q}(F)\] holds if and only if the ray $l^+_{PE}$ is
above $l^+_{PF}$. \label{lemma:slopecompare}
\end{lemma}
\begin{proof}
By the formula of $Z_{s,q}$, the angle between the
rays $l^+_{PE}$ and $l_{P-}$ at the point $P$ is
$\pi\phi_{s,q}(E)$. The statement follows from this observation.
\end{proof}


\begin{center}

\tikzset{%
    add/.style args={#1 and #2}{
        to path={%
 ($(\tikztostart)!-#1!(\tikztotarget)$)--($(\tikztotarget)!-#2!(\tikztostart)$)%
  \tikztonodes},add/.default={.2 and .2}}
}

\begin{tikzpicture}[domain=2:1]
\newcommand\XA{0.1}
\newcommand\obe{-0.3}

\coordinate (E) at (-3,-2);
\node  at (E) {$\bullet$};
\node [below] at (E) {$E$};

\coordinate (F) at (3,-1);
\node [left] at (F) {$F$};
\node at (F) {$\bullet$};

\coordinate (P) at (0.5,1.5);
\node [above] at (P) {$P$};
\node at (P) {$\bullet$};

\draw [dashed] (P) -- (F);
\node at (2,0.5) {$l^+_{PF}$};

\draw  (P) -- (E);
\draw [add =0 and 0.8,dashed] (E) to (P) node [right]{$l^+_{PE}$};




\draw[->] [opacity=\XA] (-4,0) -- (,0) node[above right] {}-- (4,0) node[above right, opacity =1] {$\frac{v_1}{v_0}$};

\draw[->][opacity=\XA] (0,-2.5)-- (0,0) node [above right] {O} --  (0,5) node[right, opacity=1] {$\frac{v_2}{v_0}$};

\draw [thick](-3,4.5) parabola bend (0,0) (3,4.5) node [above, opacity = 1] {$\frac{1}{2}s^2-q=0$};

\end{tikzpicture} 

Figure: comparing phases at $\sigma_P$.
\end{center}


\section{Smoothness of Bridgeland moduli spaces}
\subsection{Bounds on phases of stable factors}

In this section we prove a lemma on bounding phases of stable factors of a given object when deforming the stability condition. This is first proved for $\pp$ in \cite{LZ1,LZ2}, and the current version of the lemma, which works for the general situation, is suggested to us by Bayer.

\begin{lemma}[Bayer]
Suppose $P$ and $Q$ are two points in the $\cccp$ with $\bi$. Let $E$ be a $\sigma_P$-stable object in $\Coh_{P}$, $A$ and $B$ be the intersection points of the line $L_{v(E)P}$ and the parabola $q=\frac{1}{2}s^2$. Denote the stable factors  of $E$ with respect to $\sigma_Q$ by $E_i$, then for each factor, the phase $\phi_Q(E_i)$ lies in between $\phi_Q (A)$ and $\phi_Q(B)$.
\label{ph}
\end{lemma}

\begin{rem}
The phase $\phi_Q(A)$ is the real number $\frac{1}{\pi}\Arg({l^+_{PQ},l_{Q-}})$ up to an integer. It is explicitly given by 
\[\phi_P(E)+\frac{1}{\pi}\tilde{\Arg}_A({l_{AQ},l_{AP}}).\]
Here $\tilde{\Arg}_A({l_{AQ},l_{AP}})$ is the degree of the rotation from  $l_{AQ}$ to $l_{AP}$ clockwisely, and it belongs to $(-\pi,\pi)$.
\end{rem}

\begin{proof}
We will focus on the case when $P$ and $Q$ are both to the left of $v(E)$, and $Q$ lies below the line $L_{v(E)P}$. The other  cases can be proved similarly. Also assume that $A$ is the left intersection point and $B$ is the right one. 


\begin{center}

\tikzset{%
    add/.style args={#1 and #2}{
        to path={%
 ($(\tikztostart)!-#1!(\tikztotarget)$)--($(\tikztotarget)!-#2!(\tikztostart)$)%
  \tikztonodes},add/.default={.2 and .2}}
}

\begin{tikzpicture}[domain=2:1]
\newcommand\XA{0.1}
\newcommand\obe{-0.3}
\coordinate (P1) at (-2.5,3.125);
\node [below] at (P1) {$A$};
\node at (P1) {$\bullet$};

\coordinate (Q1) at (1.5,1.125);
\node [above] at (Q1) {$B$};
\node at (Q1) {$\bullet$};

\draw [add =0 and 0.5] (P1) to (Q1) node{$\bullet$} node [above]{$E$} coordinate (E);

\coordinate (P) at (0,1.875);
\node [above] at (P) {$P$};
\node at (P) {$\bullet$};

\coordinate (R) at (0.1,1);
\node [above] at (R) {$R$};
\node at (R) {$\bullet$};

\draw [add =0 and 0.8] (P) to (R) node{$\bullet$} node [below]{$Q$} coordinate (Q);
\draw [add =-1.2 and 1.1,dashed] (E) to (R) node{} node [below]{$l^-$};
\draw [add =-.2 and 0.2,dashed] (R) to (E) node{} node [below]{$l^+$};
\draw [add =-.7 and -0.3,dashed] (R) to (E) coordinate (ER) node{$\bullet$} node [below]{$E^R_i$};
\draw [add =0 and 1.2] (Q) to (Q1) node{} node [above]{$l^+_{QB}$};
\draw [add =-1 and 0.6] (P1) to (Q) node{} node [below]{$l^+_{QA}$};
\draw [add =0 and .2] (Q) to (P1) node{} node [below]{$l^-_{QA}$};
\draw [add =-1 and 1.6] (Q1) to (Q) node{} node [below]{$l^-_{QB}$};
\draw[dashed] (Q)--(ER);

\draw[->] [opacity=\XA] (-4,0) -- (,0) node[above right] {}-- (4,0) node[above right, opacity =1] {$\frac{v_1}{v_0}$};

\draw[->][opacity=\XA] (0,-1.5)-- (0,0) node [above right] {O} --  (0,4) node[right, opacity=1] {$\frac{v_2}{v_0}$};

\draw [thick](-2.7,3.645) parabola bend (0,0) (2.7,3.645) node [above, opacity =1] {$\frac{1}{2}s^2-q=0$};
\end{tikzpicture} \\
Deforming $P$ to $Q$.
\end{center}

Deform the stability condition along the line segment $l_{PQ}$. If $E$ remains stable at $\sigma_Q$, then by the picture and Lemma \ref{lemma:slopecompare}, the statement holds clearly.

If $E$ is destabilized at certain point $R$ on $l_{PQ}$, we consider any stable factor $E^R_i$ of $E$ with respect to $\sigma_R$. By Lemma \ref{lemma:paraandspanplane}, $v(E^R_i)$ lies on the line $L_{v(E)R}$. By the Bogomolov inequality, $v(E^R_i)$ lies in the region $q \leq \frac{1}{2}s^2$. So $v(E^R_i)$ lies on the two doted rays $l^-$ and $l^+$. Note that it is clear from the picture $l^-$ is contained in the region between rays $l_{QA}^-$ and $l_{QB}^-$, and $l^+$ is contained in the region between rays $l_{QA}^+$ and $l_{QB}^+$. So by Lemma \ref{lemma:slopecompare}, we know that $\phi_Q(E^R_i)$ lies between $\phi_Q (A)$ and $\phi_Q(B)$. 

Now continue deforming the stability condition along $l_{RQ}$, and repeat the argument for $E^R_i$. This completes the proof.
\end{proof}

\subsection{Smoothness}

In this section we prove the smoothness of moduli spaces of stable objects on surface $X$ whose canonical bundle has certain negativity. We fix an ample divisor $H$ and a real divisor $D$ with $K_X=tH$ and $H.D=0$. Through this section, we assume that $K^2_X<0$. Note that this condition always holds when $-K_X$ has nontrivial sections.

We first have the following lemma.

\begin{lemma}
For a stability condition $\sigma_P=\sigma_{s,q}$ with $\bi$, and a $\sigma_P$-stable object $E$, we have that
\[\Ext^2(E,E)=0.\]
\label{lemma:extvanishing}
\end{lemma}
\begin{proof}
By Serre duality, $\Ext^2(E,E)=\Hom(E,E\otimes K)$. We have that 
\[v_1(E\otimes K)=v_1(E)+H.K.\]
Since $H=tK_X$, we have
\begin{align}\label{eq2.3}
    v_2(E\otimes K)-\frac{1}{2}v_1(E\otimes K)^2=v_2(E)-\frac{1}{2}v_1(E)^2.
\end{align}
So the point $v(E\otimes K)$ is by moving $v(E)$ to the left by $-H.K$ along the parabola $q-\frac{1}{2}s^2=C$, where $C=v_2(E)-\frac{1}{2}v_1(E)^2$. 

Also move the point $P$ to the left by $-H.K$ along the parabola of the form $q-\frac{1}{2}s^2=C'$ passing through $P$, and denote this new point by $Q$. It follows from the definition of stability conditions that $E\otimes K$ is $\sigma_Q$-stable. 

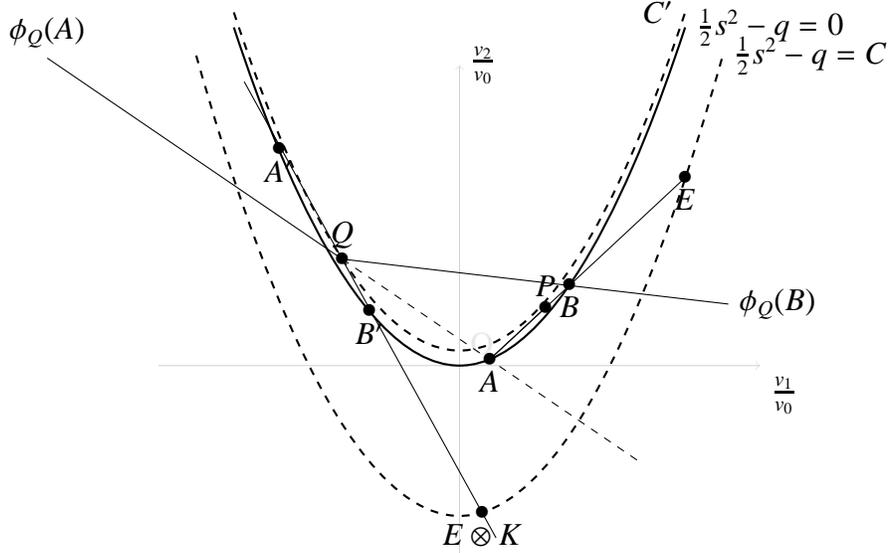
\begin{center}
\tikzset{%
    add/.style args={#1 and #2}{
        to path={%
 ($(\tikztostart)!-#1!(\tikztotarget)$)--($(\tikztotarget)!-#2!(\tikztostart)$)%
  \tikztonodes},add/.default={.2 and .2}}
}

\begin{tikzpicture}[domain=2:1]
\newcommand\XA{0.1}
\newcommand\obe{-0.3}
\newcommand\cc{0.2}

\coordinate (E) at (3,2.5);
\node  at (E) {$\bullet$};
\node [below] at (E) {$E$};

\coordinate (E1) at (0.3,-1.955);
\node  at (E1) {$\bullet$};
\node [below] at (E1) {$E\otimes K$};

\coordinate (A) at (0.4,0.08);
\node  at (A) {$\bullet$};
\node [below] at (A) {$A$};

\coordinate (B) at (1.46,1.065);
\node  at (B) {$\bullet$};
\node [below] at (B) {$B$};

\coordinate (T) at (-1.2,0.72);
\node  at (T) {$\bullet$};
\node [below] at (T) {$B'$};

\coordinate (T) at (-2.4,2.88);
\node  at (T) {$\bullet$};
\node [below] at (T) {$A'$};

\draw (A) -- (B) -- (E);

\draw [add =0 and -0.3] (A) to (B) coordinate (P) node {$\bullet$} node[above] {$P$};

\coordinate (Q) at (-1.56,1.417);
\node  at (Q) {$\bullet$};
\node [above] at (Q) {$Q$};

\draw [add =0 and 0.7] (Q) to (B)  node {} node[right] {$\phi_Q(B)$};
\draw [add =0 and 1,dashed] (Q) to (A)  node {} node[right] {};
\draw [add =-1 and 2] (A) to (Q)  node {} node[above] {$\phi_Q(A)$};
\draw [add =0.7 and 0.1] (Q) to (E1)  node {} node[right] {};


\draw[->] [opacity=\XA] (-4,0) -- (,0) node[above right] {}-- (4,0) node[below right, opacity =1] {$\frac{v_1}{v_0}$};

\draw[->][opacity=\XA] (0,-2.5)-- (0,0) node [above right] {O} --  (0,4) node[right, opacity=1] {$\frac{v_2}{v_0}$};

\draw [thick](-3,4.5) parabola bend (0,0) (3,4.5) node [right, opacity =1] {$\frac{1}{2}s^2-q=0$};
\draw [thick,dashed](-3.5,4.125) parabola bend (0,-2) (3.5,4.125) node [right, opacity =1] {$\frac{1}{2}s^2-q=C$};
\draw [thick,dashed](-3,4.5+\cc) parabola bend (0,0+\cc) (3,4.5+\cc) node [left, opacity =1] {$C'$};
\end{tikzpicture} \\
Compare the slopes of $\phi_Q(E)_{\min}$ and $\phi_Q(E\otimes K)$.
\end{center}

Now we are ready to prove the lemma. We first treat the case when $P$ is to the left of $v(E)$. Denote the intersection points of $L_{v(E)P}$ and $q-\frac{1}{2}s^2=0$ by $A$ and $B$;  and denote the intersection points of $L_{v(E\otimes K)Q}$ and $q-\frac{1}{2}s^2=0$ by $A'$ and $B'$. If the line segments $l_{AB}$ and $l_{A'B'}$ have an intersection $R$, then both $E$ and $E\otimes K$ are $\sigma_R$-stable. In addition, by \ref{lemma:slopecompare}, $\phi_R(E)>\phi_R(E\otimes K)$. Therefore,  
\[\Hom(E,E\otimes K)=0.\]
In the case that the line segments $l_{AB}$ and $l_{A'B'}$ do not intersect each other, $B'$ is to the left of $A$. It is easy to see from the picture that $\phi_Q(E\otimes K)$ is smaller than both $\phi_Q(A)$ and $\phi_Q(B)$. By Lemma \ref{lemma:extvanishing}, the stable factors of $E$ with respect to $\sigma_Q$ have phases between $\phi_Q(A)$ and $\phi_Q(B)$.  So we must have
\[\Hom(E,E\otimes K)=0.\]

If $P$ is to the right of $v(E)$, we consider the shifted derived dual $\mathbb D(E):= E^{\vee}[2]$. It is a standard result (see for example \cite{BM2}) that $\mathbb D(E)$ is stable with respect to $\sigma_{-s,q}$, with the same $H$ and $D$ replaced by $-D$. Now we reduce to the first case, and have that
\[\Hom(\mathbb D(E),\mathbb D(E)\otimes K)=\Hom(E,E\otimes K)=0.\]
In the case that  the slope of $E$ is $s$, by the locally finiteness of walls, there is an open neighborhood of $P$ such that for any $P'$ in the neighborhood, $E$ is $\sigma_{P'}$-stable. So we finish the proof.
\end{proof}

Now we can state our first main theorem.

\begin{theorem}
For a stability condition $\sigma_P=\sigma_{s,q}$ with $\bi$, the moduli stack of $\sigma_P$-stable objects with a given character is smooth.
\label{thm:smoothness}
\end{theorem}

\begin{proof}
By \cite{Inab, Lieb}, there exists a deformation theory for complexes, similar to the ordinary one for coherent sheaves. In particular, the Zariski tangent space to the moduli space at an object $E$ is given by $\Ext^1(E,E)$, and the obstructions lie in $\Ext^2(E,E)$.

By Lemma \ref{lemma:extvanishing}, we have $\Ext^2(E,E)=0$, so there exists no obstruction class. Since $E$ is in a heart with respect to a t-structure, $\Ext^i(E,E)=0$ for $i\leq -1$. Due to the argument in Lemma \ref{lemma:extvanishing}, $\phi_P(E)>\phi_P(E\otimes K)_{\max}$. By Serre duality, t $\Ext^i(E,E)\simeq (\Hom(E,E\otimes K[2-i]))^* = 0$ for  $i\geq 3$. Since $E$ is stable, we know $\hom(E,E)=1$. So $\ext^1(E,E)=1-\chi(E,E)$ only depends on the character, hence is constant over the stable locus. This proves the smoothness of the moduli stack of stable objects.
\end{proof}

\section{Poisson structures on Bridgeland moduli spaces of Poisson surfaces}
Recall that a (holomorphic) Poisson structure on a compact complex manifold $M$ is given by a bivector field $\theta\in H^0(M, \wedge^2TM)$ satisfying a closure condition. Such a $\theta$ induces a homomorphism of vector bundles $B: T^*M\to TM$, with
\[\langle \theta, \alpha\wedge\beta\rangle = \langle B(\alpha), \beta\rangle\]
for 1-forms $\alpha$, $\beta$. We define an operator $\tilde d: H^0(M, \wedge^2TX) \to H^0(M, \wedge^3TX)$ by
\begin{align*}
\tilde d\theta(\alpha,\beta,\gamma) = &B(\alpha)\theta(\beta,\gamma) + B(\beta)\theta(\gamma, \alpha) + B(\gamma)\theta(\alpha, \beta) \\
&-\langle[B(\alpha),B(\beta)],\gamma\rangle-\langle[B(\beta),B(\gamma)],\alpha\rangle-\langle[B(\gamma),B(\alpha)],\beta\rangle
\end{align*}
for 1-forms $\alpha$, $\beta$, $\gamma$, where $[\cdot,\cdot]$ is the commutator of vector fields. As stated in Proposition 1.1 in \cite{Bo}, the closure condition for $\theta$ is given by 
\[\tilde d \theta=0.\]

Now let $X$ be a smooth projective surface. Since the closure condition holds automatically, $X$ carries a non-zero (holomorphic) Poisson structure if and only if $-K_X$ has sections. Through this section we assume that $X$ is a Poisson surface and $-K_X$ is nontrivial. Moreover, we fix a Poisson structure $s\in H^0(X, -K_X)$.

Choose a geometric stability condition on $X$ as that constructed in Definition \ref{defn:geostabsq}, and let $M$ be the moduli space of semistable objects of a given character. Assume that we are in the situation of Theorem \ref{thm:smoothness}. We want to show that $M$ has a canonical Poisson structure $\theta=\theta_s$.

As shown in \cite{In}, the universal family $\mathscr E$ of $M$ exists in a local analytic neighborhood of $M\times X$. Let $p: M\times X \to M$ and $q: M\times X \to X$ be the projection maps. The relative extension sheaf $\mathcal Ext^1_{\mathcal O_M}(\mathscr E, \mathscr E)$ is independent of the choice of the universal family in local analytic neighborhood, and extends to a globally well-defined sheaf. We have the canonical identification
\[TM\cong \mathcal Ext^1_{\mathcal O_M}(\mathscr E, \mathscr E).\]
Similarly, we have
\[T^*M\cong \mathcal Ext^1_{\mathcal O_M}(\mathscr E, \mathscr E\otimes q^*K_X).\]

In order to define the Poisson structure, for any stable object $E$, consider the following map
\[\theta(E): \Ext^1(E,E\otimes K_X) \times \Ext^1(E,E\otimes K_X) \to\Ext^2(E,E\otimes K_X^2) \xrightarrow{\cdot s} \Ext^2(E,E\otimes K_X) \xrightarrow{\mathrm{Tr}} \mathbb C,\]
where the first map is given by the identification $\Ext^1(E,E\otimes K_X)\cong \Ext^1(E\otimes K_X,E\otimes K_X^2)$ and the Yoneda product, the second map is induced by tensoring $s\in H^0(X, -K_X)$, and the third map is the trace map from Serre duality.

\begin{lemma}
The map $\theta(E)$ is antisymmetric.
\label{ans}
\end{lemma}

\begin{proof}
By taking a locally free resolution $E^\bullet$ of finite length for the object $E$,  the map $\theta(E)$ is by taking the hypercohomology functor $\mathbb H$ in the degree $(1,1)$ piece on the complexes of sheaves:
\[ \mathcal Hom^\bullet(E^\bullet,E^\bullet\otimes K_X) \times \mathcal Hom^\bullet(E^\bullet,E^\bullet\otimes K_X) \xrightarrow{\circ} \mathcal Hom^\bullet(E^\bullet,E^\bullet\otimes K_X^2) \xrightarrow{\cdot s} \mathcal Hom^\bullet(E^\bullet,E^\bullet\otimes K_X) \xrightarrow{\mathrm{tr}} K_X.\]
As introduced in Chapter 10 \cite{HL}, the trace map is defined by setting $\mathrm{tr}|_{\mathcal Hom(E^i,E^j\otimes K_X)}=0$ when $i\neq j$, and $\mathrm{tr}|_{\mathcal Hom(E^i,E^i\otimes K_X)}=(-1)^i\mathrm{tr}_{E_i}$. For any homogeneous local sections $a$ and $b$ in $\mathcal Hom^\bullet(E^\bullet,E^\bullet\otimes K_X)$, we have
\[\mathrm{tr}\left((a\circ b)\cdot s\right)=(-1)^{\mathrm{deg} \; a\; \mathrm {deg}\; b}\mathrm{tr}\left((b\circ a)\cdot s\right).\]
Now as in the proof of Lemma 10.1.3 \cite{HL}, let $T:A^\bullet\otimes B^\bullet\rightarrow B^\bullet \otimes A^\bullet$ and  $T:\mathbb H(A^\bullet)\otimes \mathbb H(B^\bullet)\rightarrow \mathbb H(B^\bullet) \otimes \mathbb H(A^\bullet)$ be the twist operator $a\otimes b\mapsto (-1)^{\mathrm{deg}\; a\;\mathrm {deg} \; b} b\otimes a$ for homogeneous elements $a$ and $b$. We have the following commutative diagram:

\begin{tikzpicture}
  \matrix (m) [matrix of math nodes,row sep=3em,column sep=3em,minimum width=0em]
  {
     \H(\cH om^\bullet(E^\bullet,E^\bullet\otimes K_X))\otimes \H(\cH om^\bullet(E^\bullet,E^\bullet\otimes K_X)) & \H(\cH om^\bullet(E^\bullet,E^\bullet\otimes K_X))\otimes \H(\cH om^\bullet(E^\bullet,E^\bullet\otimes K_X)) \\
     \H(\cH om^\bullet(E^\bullet,E^\bullet\otimes K_X)\otimes \cH om^\bullet(E^\bullet,E^\bullet\otimes K_X)) & \H(\cH om^\bullet(E^\bullet,E^\bullet\otimes K_X)\otimes \cH om^\bullet(E^\bullet,E^\bullet\otimes K_X)) \\
     \mathrm{H}^2(K_X)  & \mathrm{H}^2(K_X)\\};
  \path[-stealth]
    (m-1-1) edge node [left] {$\mu$} (m-2-1)
            edge node [below] {$T$} (m-1-2)
    (m-2-1) edge node [above] {$\H(T)$} (m-2-2)
    (m-1-2) edge node [right] {$\mu$} (m-2-2)
    (m-2-1) edge node [right] {$\mathrm{tr}\left( (-\circ -)\cdot s\right)$} (m-3-1)
    (m-2-2) edge node [right] {$\mathrm{tr}\left( (-\circ -)\cdot s\right)$} (m-3-2)
    (m-3-1) edge [double] node [above] {id} (m-3-2);
\end{tikzpicture}

We may take the degree $(1,1)$ piece in the first row, then $\H^1(\cH om^\bullet(E^\bullet,E^\bullet\otimes K_X))=$Ext$^1(E,E\otimes K_X)$ and $\theta(E)$ is the composition map on each column. Since the twist operator changes the signs in this case, $\theta(E)$ is anti-symmetric.
\end{proof}

This fiber-wise defined map extended globally by the method analogous to Proposition 2.2 and 2.5 in \cite{Muk} or Proposition 4.1 in \cite{Bo}.  The associated $B: \Ext^1(E,E\otimes K_X)\to \Ext^1(E,E)$ is induced by $s\in H^0(X, -K_X)$.

\begin{theorem}
Adopt the notation as above and the smoothness assumption on $M$, then $\theta$ defines a Poisson structure on $M$.
\label{thm:poissonclosure}
\end{theorem}

\begin{proof}
We need to show the closure condition: $\tilde d \theta=0$. As this is a local condition, we only need to prove it in any open set $U$. By possibly shrinking $U$, we can assume that there exists universal family $\mathscr E$ over $U\times X$. By abuse of notations, we still let $p: U\times X \to U$ and $q: U\times X \to X$ be the projection maps. 

Let $\mathcal O(1)$ be an ample line bundle on $X$, such that $\mathcal O(1)\otimes K_X$ is also ample. Consider the ample sequence generated by $\mathcal O(1)$. We can take a resolution $V^\bullet \to \mathscr E$, where $V^\bullet=\{V_i\otimes q^*\mathcal O(-m_i), d^i\}$ and $V_i$ are vector spaces.

For a vector field $u$ on $U$, we would like to express $u$ explicitly under the isomorphism $T_U(U)\simeq \Ext^1_{\mathcal O_U}(\mathscr E, \mathscr E)$. Note that the derivation $D_u: \mathcal O_U\to \mathcal O_U$ can be canonically extended to a derivation
\[D_u: V_i^\vee \otimes V_{i+1}\otimes q^* \mathcal O_X(m_i-m_{i+1}) \otimes p^*\mathcal O_U \to V_i^\vee \otimes V_{i+1}\otimes q^* \mathcal O_X(m_i-m_{i+1}) \otimes p^*\mathcal O_U.\]
So we have a well-defined map $D_u(d^i): V_i\otimes q^*\mathcal O(-m_i) \to V_{i+1}\otimes q^*\mathcal O(-m_{i+1})$. 
It is shown in \cite{In} that
\[d^{i+1}\circ D_u(d^i) + D_u(d^{i+1})\circ d^i = 0,\]
and therefore
\[\{D_u(d^i): V_i\otimes q^*\mathcal O(-m_i) \to V_{i+1}\otimes q^*\mathcal O(-m_{i+1}) \} \in \Ext^1_{\mathcal O_U}(\mathscr E, \mathscr E).\] 

For a 1-form $\alpha$ on $U$, with the given resolution $V^\bullet \to \mathscr E$, $\alpha$ can be represented by
\[\{\alpha^i: V_i\otimes q^*\mathcal O(-m_i) \to V_{i+1}\otimes q^*\mathcal O(-m_{i+1}) \otimes q^* K_X\}.\]
Now there are two ways to represent the vector field $B(\alpha)$. On the one hand, by above discussion, $B(\alpha)$ can be represented by $\{D_{B(\alpha)}(d^i): V_i\otimes q^*\mathcal O(-m_i) \to V_{i+1}\otimes q^*\mathcal O(-m_{i+1}) \}$. On the other hand, by the definition of $B$, $B(\alpha)$ is given by $\{s\alpha^i: V_i\otimes q^*\mathcal O(-m_i) \to V_{i+1}\otimes q^*\mathcal O(-m_{i+1}) \}$. So we have
\[D_{B(\alpha)}(d^i)=s\alpha^i.\]

Now for 1-forms $\alpha$, $\beta$, $\gamma$, we have the following computation
\begin{align*}
\tilde d\theta(\alpha, \beta, \gamma)&=B(\alpha)\langle B(\beta), \gamma \rangle + B(\beta)\langle B(\gamma), \alpha \rangle + B(\gamma)\langle B(\alpha), \beta \rangle \\
&\;\;\;\;-\langle [B(\alpha), B(\beta)], \gamma \rangle -\langle [B(\beta), B(\gamma)], \alpha \rangle -\langle [B(\gamma), B(\alpha)], \beta \rangle \\
&=D_{B(\alpha)}(D_{B(\beta)}(d^{i+1})\circ \gamma^i) + D_{B(\beta)}(D_{B(\gamma)}(d^{i+1})\circ \alpha^i) + D_{B(\gamma)}(D_{B(\alpha)}(d^{i+1})\circ \beta^i)\\
&\;\;\;\;-[D_{B(\alpha)}, D_{B(\beta)}](d^{i+1})\circ \gamma^i -[D_{B(\beta)}, D_{B(\gamma)}](d^{i+1})\circ \alpha^i -[D_{B(\gamma)}, D_{B(\alpha)}](d^{i+1})\circ \beta^i \\
&=D_{B(\beta)}(d^{i+1})\circ D_{B(\alpha)}(\gamma^i) + D_{B(\gamma)}(d^{i+1})\circ D_{B(\beta)}(\alpha^i) + D_{B(\alpha)}(d^{i+1})\circ D_{B(\gamma)}(\beta^i) \\
&\;\;\;\;+ D_{B(\beta)}D_{B(\alpha)}(d^{i+1})\circ \gamma^i + D_{B(\gamma)}D_{B(\beta)}(d^{i+1})\circ \alpha^i + D_{B(\alpha)}D_{B(\gamma)}(d^{i+1})\circ \beta^i.
\end{align*}

Note that $\gamma$ is a $1$-form, we have
\[\gamma^{i+1}\circ d^i + d^{i+1} \circ \gamma^i=0,\]
Applying $D_{B(\alpha)}$, we get
\[D_{B(\alpha)}(\gamma^{i+1})\circ d^i + d^{i+1} \circ D_{B(\alpha)}(\gamma^i) + D_{B(\alpha)}(d^{i+1})\circ \gamma^i + \gamma^{i+1} \circ D_{B(\alpha)}(d^i) = 0.\]
As we have seen, $\{D_{B(\alpha)}(d^i)\} \in \Ext^1_{\mathcal O_U}(\mathscr E, \mathscr E)$. Similar to Lemma \ref{ans}, it is easy to check that
\[D_{B(\alpha)}(d^{i+1})\circ \gamma^i + \gamma^{i+1} \circ D_{B(\alpha)}(d^i) = 0.\]
Hence
\[D_{B(\alpha)}(\gamma^{i+1})\circ d^i + d^{i+1} \circ D_{B(\alpha)}(\gamma^i)=0,\]
and $\{D_{B(\alpha)}(\gamma^i)\} \in \Ext^1_{\mathcal O_U}(\mathscr E, \mathscr E\otimes q^*K_X)$.

Now we have
\begin{align*}
&\;\;\;\;D_{B(\beta)}(d^{i+1})\circ D_{B(\alpha)}(\gamma^i) + D_{B(\alpha)}D_{B(\gamma)}(d^{i+1})\circ \beta^i \\
&=s\beta^{i+1}\circ D_{B(\alpha)}(\gamma^i) + D_{B(\alpha)}(s\gamma^{i+1})\circ \beta^i \\
&=s\big(\beta^{i+1}\circ D_{B(\alpha)}(\gamma^i) + D_{B(\alpha)}(\gamma^{i+1})\circ \beta^i \big) \\
&=0.
\end{align*}
The last vanishing follows from Lemma \ref{ans} of the anti-symmetry of $\theta$. Similarly we have vanishing for the other two pairs in the last expression of $\tilde d\theta(\alpha, \beta, \gamma)$. So we prove that $\tilde d\theta(\alpha, \beta, \gamma)=0$, and $\theta$ is a Poisson structure.
\end{proof}

\bibliographystyle{abbrv}\bibliography{mmpbb}

Chunyi Li \space\space\space\space\space\space\space\space\space Email address: Chunyi.Li@ed.ac.uk 

School of Mathematics, The University of Edinburgh, James Clerk Maxwell Building, The King's Buildings, Mayfield Road, Edinburgh, Scotland EH9 3JZ, United Kingdom\\

Xiaolei Zhao \space\space\space\space\space\space\space\space Email address: X.Zhao@neu.edu 

Department of Mathematics, Northeastern University, 360 Huntington Avenue, Boston, MA 02115, USA

\end{document}